\documentclass[10pt,twosude,leqno]{siamltex}
\usepackage{amsfonts,epsfig}
\usepackage{amsmath}
\usepackage{amssymb}

\def\cvd{~\vbox{\hrule\hbox{%
     \vrule height1.3ex\hskip0.8ex\vrule}\hrule } }

\begin{document}
\bibliographystyle{plain}

\title{Perturbation analysis of $A_{T,S}^{(2)}$ on Banach spaces}
\author{Fapeng Du\thanks{Department of mathematics, East China Normal University,
  Shanghai 200241, P.R. China and School of Mathematical \& Physical Sciences,
  Xuzhou Institute of Technology, Xuzhou 221008, Jiangsu Province, P.R. China (E-mail: jsdfp@163.com)}
 \and
 Yifeng Xue\thanks{Department of mathematics, East China Normal University, Shanghai 200241, P.R. China
 (Corresponding author, E-mail: yfxue@math.ecnu.edu.cn)}
}

\pagestyle{myheadings}
\markboth{Fapeng Du and Yifeng Xue}{Perturbation analysis of $A_{T,S}^{(2)}$ on Banach spaces}
 \maketitle

\begin{abstract}
In this paper, the perturbation problems of $A_{T,S}^{(2)}$ are considered. By virtue of the gap between subspaces,
we derive the conditions that make the perturbation of $A_{T,S}^{(2)}$ is stable when $T,S$ and $A$  have suitable
perturbations. At the same time, the explicit formulas for perturbation of $A_{T,S}^{(2)}$ and new results on perturbation
bounds are obtained.

\end{abstract}

\section{Introduction}

In recent years, there are many fruitful results concerning the quantitative analysis of the perturbation of the
Moore--Penrose inverses on Hilbert spaces and Drazin inverses on Banach spaces. For example, G. Chen, M. Wei and
Y. Xue gave an estimation of the perturbation bounds of the Moore--Penrose inverse on Hilbert spaces under stable
perturbation of operators, which is a generalization of the rank--preserving perturbation of matrices in \cite{CWX, XC1};
Meanwhile, many perturbation analysis results of the Drazin inverse on Banach spaces have been obtained in
\cite{GK, GKR, GKW} and \cite{K} respectively by means of the gap--function.
Recently, G. Chen and Y. Xue gave some estimations of the perturbations of the Drazin inverse on a Banach space and a
Banach algebra in \cite{Xue} and \cite{XC2} respectively under stable perturbations.

Let $X,\,Y$ be Banach spaces and let $B(X,Y)$ denote the set of bounded linear operators from $X$ to $Y$. For an operator
$A\in B(X,Y)$, let $R(A)$ and $N(A)$ denote the range and kernel of $A$, respectively. Let $T$ be a closed subspace of
$X$ and $S$ be a closed subspace of $Y$. Recall that $A_{T,S}^{(2)}$ is the unique
operator $G$ satisfying
\begin{equation}\label{1eqa}
GAG=G,\quad R(G)=T,\quad N(G)=S.
\end{equation}
It is known that (\ref{1eqa}) is equivalent to the following
condition:
\begin{equation}\label{1eqb}
N(A)\cap T=\{0\},\quad AT\dotplus S=Y
\end{equation}
(cf. \cite{DS, DSW}). It is well--known that the commonly five kinds of generalized inverses: the Moore--Penrose inverse
$A^+$, the weighted Moore--Penrose inverse $A^+_{MN}$, the Drazin inverse $A^D$, the group inverse $A^g$ and the
Bott--Duffin inverse $A^{(-1)}_{(L)}$ can be reduced to an $A_{T,S}^{(2)}$ for certain choices of $T$ and $S$ (cf.
\cite{Dj, DS, DSW}).

The perturbation analysis of $A_{T,S}^{(2)}$ have been studied by several authors (see \cite{WW, WG}, \cite{ZW1, ZW2})
when $X$ and $Y$ are of finite--dimensional. A lot of results about the error bounds have been obtained. But when $X$
and $Y$ are of infinite--dimensional, there is little known about the perturbation of $A_{T,S}^{(2)}$ if $T$, $S$ and $A$
have small perturbations respectively. In this paper, using the gap--function $\hat\delta(\cdot,\cdot)$ of two closed
subspaces, we give the the upper bounds of $\|\bar A_{T^\prime,S^\prime}^{(2)}\|$ and
$\|\bar A_{T^\prime,S^\prime}^{(2)}-A_{T,S}^{(2)}\|$ respectively. The main result is the following:

Let $A,\bar{A}=A+E\in B(X,Y)$ and let $T\subset X$, $S\subset Y$ be closed subspaces such that $A_{T,S}^{(2)}$ exists.
Let $T'\subset X$, $S'\subset Y$ be closed subspaces with $\hat{\delta}(T,T^\prime)< \dfrac{1}{(1+ \kappa )^2}$ and
$\hat{\delta}(S,S^\prime)< \dfrac{1}{3+ \kappa }$. Suppose that
$\|A_{T,S}^{(2)}\|\|E\|<\dfrac{2\kappa}{(1+\kappa)(4+\kappa)}$. Then $\bar{A}_{T^\prime,S^\prime}^{(2)}$ exists and
\begin{align*}
\|\bar{A}_{T^\prime,S^\prime}^{(2)}\|&\leq \frac{(1+\hat{\delta} (S^\prime,S))\|A^{(2)}_{T,S}\|}
{1-(1+ \kappa )\hat{\delta}(T,T^\prime )- \kappa \hat{\delta} (S^\prime,S)-
(1+\hat{\delta}(S^\prime,S))\|A^{(2)}_{T,S}\|\|E\|}.\\
\frac{\|\bar{A}_{T^\prime,S^\prime}^{(2)}-A_{T,S}^{(2)}\|}{\|A^{(2)}_{T,S}\|}&\leq \frac{(1+ \kappa )
(\hat{\delta}(T,T^\prime)+\hat{\delta} (S^\prime,S)+(1+\hat{\delta}(S^\prime,S))\|A^{(2)}_{T,S}\|\|E\|}
{1-(1+ \kappa )\hat{\delta}(T,T^\prime )- \kappa \hat{\delta} (S^\prime,S)-(1+\hat{\delta} (S^\prime,S))\|A^{(2)}_{T,S}\|
\|E\|},
\end{align*}
where $\kappa=\|A\|\|A_{T,S}^{(2)}\|$ is called the condition number of $A_{T,S}^{(2)}$. These results improve
Theorem 4.4.5 of \cite{XUE}.

\section{Preliminaries}
\setcounter{equation}{0}

Let $Z$ be a complex Banach space. Let $M,\,N$ be two closed
subspaces in $Z$. Set
$$
\delta(M,N)=\begin{cases}\sup\{dist(x,N)\,\vert\,x\in M,\,\|x\|=1\},\quad &M\not=\{0\}\\ 0 \quad& M=\{0\}
\end{cases},
$$
where $dist(x,N)=\inf\{\|x-y\|\,\vert\,y\in N\}$. The gap $\hat\delta(M,N)$ of $M,\,N$ is given by
$\hat{\delta}(M,N)=\max\{\delta(M,N),\delta(N,M)\}$. For convenience, we list some properties about $\delta(M,N)$
and $\hat\delta(M,N)$ which come from \cite{TK} as follows.

\begin{proposition}\label{2P1}
Let $M,\,N$ be closed subspaces in a Banach space $Z$. Then
\begin{enumerate}
  \item[$(1)$] $\delta(M,N)=0$ if and only if $M\subset N$.
  \item[$(2)$] $\hat{\delta}(M,N)=0$ if and only if $M=N$.
  \item[$(3)$] $\hat{\delta}(M,N)=\hat{\delta}(N,M)$.
  \item[$(4)$] $0\leq \delta(M,N)\leq 1$, $0\leq \hat{\delta}(M,N)\leq 1$.
\end{enumerate}
\end{proposition}

An operator $A\in B(Z,Z)$ is group invertible if there is $B\in B(Z,Z)$ such that
$$
ABA=A,\quad BAB=B,\quad AB=BA.
$$
The operator $B$ is called the group inverse of $A$ and is denoted by $A^g$. Clearly, $R(A^g)=R(A)$ and $N(A^g)=N(A)$.

\begin{lemma}\label{2L1}
Let $A \in B(X,Y)$. Let $T\subset X$, $S\subset Y$ be closed subspaces such that $A_{T,S}^{(2)}$ exists.
Let $G\in B(Y,X)$ be an operator with $R(G)=T$ and $N(G)=S$. Then
\begin{enumerate}
\item[$(1)$] $R(AG)=AT$, $N(AG)=S$ and $R(GA)=T$, $N(GA)\cap T=\{0\}$.
\item[$(2)$] $GA$ and $AG$ are group invertible and $A_{T,S}^{(2)}=(GA)^gG=G(AG)^g$.
\end{enumerate}
\end{lemma}
\begin{proof} (1) Using $AT\dotplus S=Y$ and $N(A)\cap T=\{0\}$, we can obtain the assertion.

(2) The assertion follows from \cite[Lemma 3.1]{Dj}.
\end{proof}

\begin{lemma}[{\cite[Theorem 11, pp. 100]{VM}}]\label{2L2}
Let $M$ be a complemented subspace of $X$. Let $P\in B(X,X)$ be an
idempotent operator with $R(P)=M$. Let $M^\prime$ be a closed
subspace of $H$ satisfying
$\hat{\delta}(M,M^\prime)<\dfrac{1}{1+\|P\|}$. Then $M^\prime$ is
complemented, that is, $H=R(I-P)\dotplus M'$.
\end{lemma}

Let $A \in B(X,Y)$. Let $T\subset X$ and $S\subset Y$ be closed subspaces such that $A^{(2)}_{T,S}$ exists.
Put $\kappa=\|A\|\|A^{(2)}_{T,S}\|$. The symbol $\kappa$ will be used throughout the paper.

\begin{lemma}\label{2L3}
Let $A \in B(X,Y)$. Let $T\subset X$ and $S\subset Y$ be closed subspaces such that $A_{T,S}^{(2)}$ exists.
Let $T'$ be a closed subspace of $X$ such that $\hat{\delta}(T,T^\prime)< \dfrac{1}{1+\kappa}$.
Then
\begin{enumerate}
\item[$(1)$] $\hat\delta(AT,AT^\prime)\leq \dfrac{\kappa\,\hat\delta(T,T^\prime)}{1-(1+\kappa)\hat\delta(T,T^\prime)}$.
\item[$(2)$] $N(A)\cap T'=\{0\}$.
\end{enumerate}
\end{lemma}
\begin{proof} (1) First we show $\delta(AT,AT^\prime)\leq\|A\|\|A_{T,S}^{(2)}\| \delta(T,T^\prime)\leq\kappa\,
\hat\delta(T,T^\prime)$.

Let $x\in T$. Then $x=A_{T,S}^{(2)}Ax$ and $\|x\|\leq\|A_{T,S}^{(2)}\|\|Ax\|$. For any $y\in T^\prime$, we have
$\|Ax-Ay\|\leq \|A\|\|x-y\|$. So
\begin{align*}
dist(Ax,AT^\prime)&=\inf_{y\in T^\prime}\|Ax-Ay\|\leq \|A\|\inf_{y\in T^\prime}\|x-y\|\\
&=\|A\|dist(x,T^\prime)\leq\|A\|\|x\|\delta(T,T')\\
&\leq\|A\|\|A_{T,S}^{(2)}\|\|Ax\|dist(T,T').
\end{align*}
This means that $\delta(AT,AT^\prime)\leq \kappa\,\delta(T,T^\prime)\leq\kappa\,\hat\delta(T,T^\prime)$.

Next we show $\delta(AT^\prime,AT)\leq \dfrac{\kappa\,\hat\delta(T,T')}{1-(1+\kappa)\hat\delta(T,T')}$
when $\hat{\delta}(T,T^\prime)< \dfrac{1}{1+\kappa}.$

For each $x^\prime \in T^\prime$ and $x\in T$, we have
\begin{align*}
\|Ax^\prime\|&=\|A(x^\prime-x+x)\|\geq\|Ax\|-\|A\|\|x^\prime-x\| \\
&\geq \|A_{T,S}^{(2)}\|^{-1}\|x\|-\|A\|\|x^\prime-x\| \\
&\geq \|A_{T,S}^{(2)}\|^{-1}\|x^\prime\|-\|A_{T,S}^{(2)}\|^{-1}\|x^\prime-x\|-\|A\|\|x^\prime-x\|\\
&\geq \|A_{T,S}^{(2)}\|^{-1}\|x^\prime\|-(\|A_{T,S}^{(2)}\|^{-1}+\|A\|)\|x^\prime-x\|,
\end{align*}
Thus,
$$
(\|A_{T,S}^{(2)}\|^{-1}+\|A\|)\|x^\prime-x\|\geq \|A_{T,S}^{(2)}\|^{-1}\|x^\prime\|-\|Ax^\prime\|
$$
and consequently,
$$
\|A_{T,S}^{(2)}\|^{-1}\|x^\prime\|-\|Ax^\prime\|\leq\|x^\prime\|(\|A_{T,S}^{(2)}\|^{-1}+\|A\|)\delta(T',T),
$$
that is,
\begin{equation}\label{2eqa}
\|A_{T,S}^{(2)}\|\|Ax^\prime\|\geq \big[1-(1+\|A\|\|A_{T,S}^{(2)}\|)\delta(T',T)\big]\|x'\|.
\end{equation}
Therefore,
\begin{align*}
dist(Ax^\prime, AT)&\leq \|A\|dist(x^\prime,T)\leq\|A\|\|x'\|\delta(T',T)\\
&\leq\frac{\|A\|\|Ax'\|\|A_{T,S}^{(2)}\|\hat\delta(T,T')}{1-(1+\|A\|\|A_{T,S}^{(2)}\|)\hat\delta(T,T')},
\end{align*}
i.e., $\delta(AT',AT)\leq\dfrac{\kappa\,\hat\delta(T,T')}{1-(1+\kappa)\hat\delta(T,T')}$
when $\hat{\delta}(T,T^\prime)<\dfrac{1}{1+\kappa}.$

The final assertion follows from above arguments.

(2) From (\ref{2eqa}), we get that $N(A)\cap T'=\{0\}$.
 \end{proof}

\section{Main results}
\setcounter{equation}{0}
.
\begin{lemma}\label{3L1}
Let $A \in B(X,Y)$ and let $T\subset X$, $S\subset Y$ be closed subspaces such that $A_{T,S}^{(2)}$ exists.
Let $T'$ be closed subspace in $X$ with $\hat\delta(T,T^\prime)<\dfrac{1}{(1+ \kappa )^2}$. Then $A_{T^\prime,S}^{(2)}$
exists and
\begin{enumerate}
\item[$(1)$] $A_{T^\prime,S}^{(2)}=A_{T,S}^{(2)}+(I-A_{T,S}^{(2)}A)F(I+(AG)^gAF)^{-1}(AG)^g$, where
$G,\,H \in B(Y,X)$ are arbitrary operators such that
$$
R(G)=T,\ R(H)=T^\prime,\ N(G)=N(H)=S\ \text{and}\ F=H-G.
$$
\item[$(2)$] $\|A^{(2)}_{T^\prime ,S}-A^{(2)}_{T,S}\|\leq
       \dfrac{(1+ \kappa )\hat{\delta}(T,T^\prime )}{1-(1+ \kappa )\hat{\delta}(T,T^\prime)}\|A^{(2)}_{T,S}\|$.
\item[$(3)$] $\|A^{(2)}_{T^\prime,S}\|\leq \dfrac{\|A^{(2)}_{T,S}\|}{1-(1+ \kappa )\hat{\delta}(T,T^\prime)}$.
\end{enumerate}
\end{lemma}
{\em Proof.}
Put $P_{AT,S}=AA_{T,S}^{(2)}$. Then $P_{AT,S}$ is an idempotent operator onto $AT$ along $S$. By Lemma \ref{2L3} (1),
we have
$$
\hat{\delta}(AT,AT^\prime)\leq \frac{ \kappa\, \hat{\delta}(T,T^\prime)}{1-(1+ \kappa )\hat{\delta}(T,T^\prime)}
< \frac{1}{1+ \kappa }\leq \frac{1}{1+\|P_{AT,S}\|},
$$
when $\hat{\delta}(T,T^\prime)< \dfrac{1}{(1+ \kappa )^2}.$ So $AT^\prime$ is complemented and
$AT^\prime \dotplus S=Y$ by Lemma \ref{2L2}. Consequently, $A_{T^\prime,S}^{(2)}$ exists by Lemma \ref{2L3} (2).

Let $G,\,H\in B(Y,X)$ with $R(G)=T$, $N(G)=N(H)=S$ and $R(H)=T'$. Then by Lemma \ref{2L1}, we have
$$
A_{T,S}^{(2)}=G(AG)^g=(GA)^gG,\quad A_{T',S}^{(2)}=H(AH)^g=(HA)^gH.
$$
Put $F=H-G$. Then $S\subseteq N(F)$.

Now we show that $I+(AG)^gAF$ is invertible. Let $y\in N(I+(AG)^gAF)$.  Then
$$
y=-(AG)^gAFy \in R((AG)^g)=R(AG)=AT.
$$
Hence
$$
AA_{T,S}^{(2)}y=y=-(AG)^gAFy=AA_{T,S}^{(2)}y-(AG)^gAHy.
$$
So $(AG)^gAHy=0$. This indicates that
$$
AHy \in R(AH) \cap N((AG)^g)=AT^\prime \cap S =\{0\}.
$$
From $AHy=0$, we get that $y \in N(AH)\cap AT=S \cap AT=\{0\}$, i.e. $y=0$.
Therefore $I+(AG)^gAF$ is injective.

Note that $N((AG)^g)=S$ and $AT'\dotplus S=Y$. So
$$
AT=R(AG)=R((AG)^g)=(AG)^gAT'=R((AG)^gAH)
$$
and consequently, for any $y \in Y=S\dotplus AT$, there is $y_1 \in S$ and $y_2 \in R((AG)^gAH)$ such that
$y=y_1+y_2$. Choose $z\in Y$ such that $y_2 =(AG)^gAHz$. Write $z=z_1+z_2$ where $z_1\in AT$ and $z_2\in S$.
Since $N(H)=S$, $y_2 =(AG)^gAHz_1$. Set $\xi=y_1+z_1$. Then
\begin{align*}
(I+(AG)^gAF)\xi&=(I-AA_{T,S}^{(2)}+(AG)^gAH)\xi \\
&=(I-AA_{T,S}^{(2)})\xi+(AG)^gAH\xi=y_1 +(AG)^gAHz_1\\
&=y_1+y_2=y,
\end{align*}
that is, $I+(AG)^gAF$ is surjective. Therefore, $I+(AG)^gAF$ is invertible and $I+AF(AG)^g$ is invertible too.

Put
$$
D=A_{T,S}^{(2)}+(I-A_{T,S}^{(2)}A)F(I+(AG)^gAF)^{-1}(AG)^g.
$$
It is easy to verify that $DAD=D$ and $N(D)=S$. Since $(I + (AG)^gAF)^{-1}(AG)^g=(AG)^g(I + AF(AG)^g)^{-1}$ and
\begin{align*}
D&=A_{T,S}^{(2)}+(I-A_{T,S}^{(2)}A)F(I+(AG)^gAF)^{-1}(AG)^g\\
&=G(AG)^g+(I-G(AG)^gA)F(I+(AG)^gAF)^{-1}(AG)^gAG(AG)^g\\
&=(G+G(AG)^gAF+F-G(AG)^gAF)(I+(AG)^gAF)^{-1}(AG)^g\\
&=H(I+(AG)^gAF)^{-1}(AG)^g\\
&=H(AG)^g(I+AF(AG)^g)^{-1}
\end{align*}
by Lemma \ref{2L1} (2), we have that
$$
R(D)=R(H(AG)^g)=H(AT)=H(AT \dotplus S)=R(H)=T^\prime.
$$
Thus $A^{(2)}_{T^\prime ,S}=D$.

Put $W=A^{(2)}_{T^\prime ,S}-A^{(2)}_{T ,S}$.  For any $\xi \in Y=AT'\dotplus S$, there is $u\in AT'$ and $u'\in S$
such that $\xi = u+u'$. Choose $x\in Y$ such that $u=AA^{(2)}_{T',S}x$. Since
$dist(A^{(2)}_{T',S}x, T)\leq\|A^{(2)}_{T^\prime,S}x\|\delta(T',T)$, for every $\epsilon>0$,
we can find $y\in Y$ such that
$$
\|A^{(2)}_{T^\prime,S}x-A^{(2)}_{T,S}y \|<\|A^{(2)}_{T',S}x\|\delta(T',T)+\epsilon.
$$
Set $v=AA^{(2)}_{T,S}y $. Then
$$
\|u-v\|=\|AA^{(2)}_{T^\prime,S}x - AA^{(2)}_{T,S}y\|<\|A\|\|A^{(2)}_{T^\prime,S}x\|\delta(T',T)+\|A\|\epsilon.
$$
Consequently,
\begin{align}
\|W\xi\|&=\|Wu\|=\|A^{(2)}_{T^\prime,S}u- A^{(2)}_{T,S}u\|\notag \\
&\leq \|A^{(2)}_{T^\prime,S}u- A^{(2)}_{T,S}v\|+\| A^{(2)}_{T,S}u- A^{(2)}_{T,S}v\|\notag \\
&\leq \|A^{(2)}_{T^\prime,S}x-A^{(2)}_{T,S}y \| + \|A^{(2)}_{T,S}\|\|u-v\|\notag \\
\label{3eqa}& \leq(1+ \kappa )\|A^{(2)}_{T^\prime,S}x\|\delta(T',T)+(1+\kappa)\epsilon.
\end{align}
Since
\begin{equation}\label{3eqb}
\|A^{(2)}_{T^\prime,S}x\|=\|A^{(2)}_{T^\prime,S}u\|=\|W\xi +A^{(2)}_{T,S}\xi\|\leq \|A^{(2)}_{T,S}\|\|\xi\| +\|W\xi\|,
\end{equation}
it follows from (\ref{3eqa}) and (\ref{3eqb}) that
$$
\|W\xi\|\leq (1+ \kappa )(\|A^{(2)}_{T,S}\|\|\xi\| +\|W\xi\|)\delta(T',T)+(1+\kappa)\epsilon.
$$
and hence $\|W\xi\|\leq \dfrac{(1+ \kappa )\delta(T',T)}{1-(1+ \kappa )\delta(T',T)}\|A^{(2)}_{T,S}\|\|\xi\|$
by letting $\epsilon\to 0^+$. Therefore,
$$
\|A^{(2)}_{T^\prime ,S}-A^{(2)}_{T ,S}\|\leq
\dfrac{(1+ \kappa )\hat\delta(T',T)}{1-(1+ \kappa )\hat\delta(T',T)}\|A^{(2)}_{T,S}\|.
$$
Furthermore,
\begin{align*}
\|A^{(2)}_{T^\prime,S}\|&=\|W+A^{(2)}_{T,S}\|\leq \|W\|+\|A^{(2)}_{T,S}\| \\
&\leq \frac{(1+ \kappa )\hat{\delta}(T^\prime , T)}{1-(1+ \kappa )\hat{\delta}(T^\prime , T)}\|A^{(2)}_{T,S}\| +\|A^{(2)}_{T,S}\| \\
&=\frac{\|A^{(2)}_{T,S}\|}{1-(1+ \kappa )\hat{\delta}(T^\prime , T)}.~~~~~~\cvd
\end{align*}

\begin{lemma}\label{3L2}
Let $A \in B(X,Y)$ and let $T\subset X$, $S\subset Y$ be closed subspaces such that $A_{T,S}^{(2)}$ exists.
Let $S'$ be a closed subspace in $Y$ such that $\hat{\delta}(S,S')< \dfrac{1}{2+ \kappa }$. Then $A_{T,S^\prime}^{(2)}$ exists
and
\begin{enumerate}
\item[$(1)$] $A_{T,S^\prime}^{(2)}=A_{T,S}^{(2)}+A_{T,S}^{(2)}(I+(AG)^gAF)^{-1}(AG)^gAF(I-AA_{T,S}^{(2)})$, where
$F=H-G$ and $G,\,H \in B(Y,X)$ are arbitrary with $R(G)=R(H)=T$, $N(G)=S$ and $N(H)=S^\prime$.
\item[$(2)$] $\|A_{T,S}^{(2)}-A_{T,S^\prime}^{(2)}\|\leq \dfrac{(1+ \kappa )\hat{\delta} (S^\prime,S)}
{1- \kappa\, \hat{\delta} (S^\prime,S)  }\|A_{T,S}^{(2)}\|$.
\item[$(3)$] $\|A_{T,S^\prime}^{(2)}\|\leq \dfrac{1+\hat{\delta }(S^\prime,S)}{1-\kappa\,\hat{\delta} (S^\prime,S) }
\|A_{T,S}^{(2)}\|$.
\end{enumerate}
\end{lemma}
{\em Proof.}
Let $P_{S,AT}=I-AA_{T,S}^{(2)}$ be an idempotent operator from $Y$ onto $S$ along $AT$. Since
$\|P_{S,AT}\|\leq 1+\|A\|\|A_{T,S}^{(2)}\|=1+\kappa$, we have $\hat\delta(S,S^\prime)\leq \dfrac{1}{1+\|P_{S,AT}\|}.$
So $Y=AT\dotplus S'$ by Lemma \ref{2L2}. Noting that $N(A)\cap T=\{0\}$, we get that $A_{T,S^\prime}^{(2)}$ exists.

Using the facts:
$$
AT\dotplus S=Y=AT\dotplus S',\quad N(A)\cap T=\{0\}
$$
and the similar method appeared in the proof of Lemma \ref{3L1}, we can deduce that $I+(AG)^gAF$ is invertible and
so is the operator $I+AF(AG)^g$.

Put $D=A_{T,S}^{(2)}(I+(AG)^gAF)^{-1}(AG)^gAH$. Then $R(D)\subset T$, $S'\subset N(D)$ and
\begin{align}
A_{T,S}^{(2)}&+A_{T,S}^{(2)}(I+(AG)^gAF)^{-1}(AG)^gAF(I-AA_{T,S}^{(2)})\notag\\
&=A_{T,S}^{(2)}+A_{T,S}^{(2)}(I+(AG)^gAF)^{-1}[I+(AG)^gAF-I](I-AA_{T,S}^{(2)})\notag\\
&=A_{T,S}^{(2)}(I+(AG)^gAF)^{-1}[I+(AG)^gAF-(I-AA_{T,S}^{(2)})]\notag\\
\label{3eqc}&=A_{T,S}^{(2)}(I+(AG)^gAF)^{-1}(AG)^gAH.
\end{align}
Clearly $DAD=D$ by (\ref{3eqc}). In order to obtain $A_{T,S}^{(2)}=D$, we need only to prove that $T\subset R(D)$ and
$S'\supset N(D)$.

Since $AT\dotplus S=Y$ and $N((AG)^g)=S$, $R(H)=T$, it follows that
$$
R((AG)^g)=(AG)^gAT=R((AG)^gAH)
$$
and hence
\begin{align*}
R(D)&=A_{T,S}^{(2)}(I+(AG)^gAF)^{-1}R((AG)^g)=R(A_{T,S}^{(2)}(AG)^g(I+AF(AG)^g)^{-1})\\
&=A_{T,S}^{(2)}R((AG)^g)=A_{T,S}^{(2)}AT=A_{T,S}^{(2)}(AT\dotplus S)=R(A_{T,S}^{(2)})=T.
\end{align*}

Now let $x\in N(D)$ and put $y=(I+(AG)^gAF)^{-1}(AG)^gAHx$. Then $y\in S$ and $y\in R((AG)^g)=AT$. So $y=0$ and
consequently, $(AG)^gAHx=0$. But this means that $AHx\in AT\cap N((AG)^g)=AT\cap S=\{0\}$. Thus $AHx=0$ and $Hx=0$.
Since $N(A)\cap T=\{0\}$, it follows that $x\in N(H)=S'$. Therefore,
$$
A_{T,S}^{(2)}=A_{T,S}^{(2)}(I+(AG)^gAF)^{-1}(AG)^gAH.
$$

Put $B^\prime=I-AA_{T,S^\prime}^{(2)},\ B=I-AA_{T,S}^{(2)}$. Note that
$$
W=A_{T,S}^{(2)}-A_{T,S^\prime}^{(2)}=A_{T,S}^{(2)}-A_{T,S}^{(2)}AA_{T,S^\prime}^{(2)}=
A_{T,S}^{(2)}(AA_{T,S}^{(2)}-AA_{T,S^\prime}^{(2)}).
$$
So $W=A_{T,S}^{(2)}(AA_{T,S}^{(2)}-AA_{T,S^\prime}^{(2)})=A_{T,S}^{(2)}(B^\prime-B)$. Since $B^\prime \xi \in S^\prime$,
$\forall\,\xi\in Y$, we have $dist(B'\xi,S)\leq\delta(S',S)\|B'\xi\|$. Thus, for any $\epsilon>0$, there is
$u\in Y$ such that $\|B^\prime \xi-Bu \|\leq \delta(S',S)\|B^\prime \xi\|+\epsilon$ and so that
$$
\|A_{T,S}^{(2)}(B^\prime \xi-Bu )\|\leq \delta(S',S)\|B^\prime \xi\|\|A_{T,S}^{(2)}\|+\|A_{T,S}^{(2)}\|\epsilon.
$$
Noting that $A_{T,S}^{(2)}B=0$, we have
\begin{align*}
\|W\xi \|&=\|A_{T,S}^{(2)}(B^\prime \xi -B \xi)\|=\|A_{T,S}^{(2)}(B^\prime \xi -Bu)\|\\
&\leq\delta(S',S)\|B^\prime \xi\|\|A_{T,S}^{(2)}\|+\|A_{T,S}^{(2)}\|\epsilon.
\end{align*}
But $\|B^\prime \xi\|\leq \|\xi\|+\|A\|\|A_{T,S}^{(2)}\xi-W\xi\|\leq (1+ \kappa )\|\xi\|+\|A\|\|W\xi\|.$ Thus,
\begin{equation}\label{3eqd}
\|W\xi\|\leq \delta(S',S) \|A_{T,S}^{(2)}\|((1+ \kappa )\|\xi\|+\|A\|\|W\xi\|)+\|A_{T,S}^{(2)}\|\epsilon.
\end{equation}
(\ref{3eqd}) indicates that
$
\|W\|\leq \dfrac{(1+ \kappa )\hat{\delta} (S^\prime,S) }{1- \kappa \hat{\delta} (S^\prime,S)  }\|A_{T,S}^{(2)}\|
$
and
$$
\|A_{T,S^\prime}^{(2)}\|\leq \|A_{T,S}^{(2)}\|+\|W\|\leq \frac{1+\hat{\delta }(S^\prime,S)}
{1-\kappa\hat{\delta} (S^\prime,S) }\|A_{T,S}^{(2)}\|.~~~~~~\cvd
$$

We now present our main result of this paper as follows.
\begin{theorem}\label{T3}
Let $A\in B(X,Y)$ and let $T\subset X$, $S\subset Y$ be closed subspaces such that $A_{T,S}^{(2)}$ exists.
Let $T'\subset X$, $S'\subset Y$ be closed subspaces such that
$\hat{\delta}(T,T^\prime)< \dfrac{1}{(1+ \kappa )^2}$ and $\hat{\delta}(S,S^\prime)< \dfrac{1}{3+ \kappa }$
respectively. Then $A_{T^\prime,S^\prime}^{(2)}$ exists and
\begin{align*}
&\begin{aligned}
 (1) \quad A_{T^\prime,S^\prime}^{(2)}&=A_{T,S}^{(2)}+(I-A_{T,S}^{(2)}A)F(I+(AG)^gAF)^{-1}(AG)^g \\
&+\{A_{T,S}^{(2)}+(I-A_{T,S}^{(2)}A)F(I+(AG)^gAF)^{-1}(AG)^g\}\\
& \times(I+(A\tilde{G})^gA\tilde{F})^{-1}(A\tilde{G})^gA\tilde{F}(I-AA_{T,S}^{(2)})(I+AF(AG)^g)^{-1}.
\end{aligned}\\
&(2) \quad \|A_{T^\prime,S^\prime}^{(2)}-A_{T,S}^{(2)}\| \leq \frac{(1+ \kappa )(\hat{\delta}(T,T^\prime)+\hat{\delta} (S^\prime,S)}
{1-(1+ \kappa )\hat{\delta}(T,T^\prime )- \kappa \hat{\delta} (S^\prime,S)}\|A^{(2)}_{T,S}\|.\\
&(3) \quad \|A_{T^\prime,S^\prime}^{(2)}\|\leq \frac{1+\hat{\delta} (S^\prime,S)}
{1-(1+ \kappa )\hat{\delta}(T,T^\prime )- \kappa \hat{\delta} (S^\prime,S)}\|A^{(2)}_{T,S}\|.
\end{align*}
where $G,\tilde{G},\tilde{H} \in B(Y,X)$ are such that $R(G)=T$, $R(\tilde{G})=R(\tilde{H})=T^\prime$, $N(G)=N(\tilde{G})=S$,
$N(\tilde{H})=S^\prime$ and $F=\tilde{G}-G$, $\tilde{F}=\tilde{H}-\tilde{G}$.
\end{theorem}

{\em Proof.}
Since $\hat{\delta}(T,T^\prime)< \dfrac{1}{(1+ \kappa )^2}$, it follows from Lemma \ref{3L1} that $A_{T^\prime,S}^{(2)}$
exists and
\begin{align*}
\|A\|\|A^{(2)}_{T^\prime,S}\|&\leq \frac{ \kappa }{1-(1+ \kappa )\hat{\delta}(T,T^\prime)}<1+ \kappa\\
\hat{\delta}(S,S^\prime)&<\frac{1}{2+1+ \kappa }<\frac{1}{2+\|A\|\|A_{T^\prime,S}^{(2)}\|}.
\end{align*}
Thus, by Lemma \ref{3L2} we have that $A_{T^\prime,S^\prime}^{(2)}$ exists and
\begin{equation}\label{3eqe}
A_{T',S'}^{(2)}=A_{T',S}^{(2)}+A_{T',S}^{(2)}(I+(A\tilde{G})_gA\tilde{F})^{-1}(A\tilde{G})_gA\tilde{F}(I-AA_{T',S}^{(2)})
\end{equation}
by Lemma \ref{3L2}, where $\tilde G$, $\tilde H\in B(Y,X)$ with $R(\tilde{G})=T'$, $N(\tilde{G})=S$ and $R(\tilde{H})=T'$, $N(\tilde{H})=S'$
and $\tilde{F}=\tilde{H}-\tilde{G}$. By Lemma \ref{3L1}, we have
\begin{equation}\label{3eqf}
A_{T^\prime,S}^{(2)}=A_{T,S}^{(2)}+(I-A_{T,S}^{(2)}A)F(I+(AG)^gAF)^{-1}(AG)^gAA_{T,S}^{(2)}.
\end{equation}
Combining (\ref{3eqe}) with (\ref{3eqf}), we get that
\begin{align*}
A_{T^\prime,S^\prime}^{(2)}&=A_{T,S}^{(2)}+(I-A_{T,S}^{(2)}A)F(I+(AG)^gAF)^{-1}(AG)^g \\
&\ \,+\{A_{T,S}^{(2)}+(I-A_{T,S}^{(2)}A)F(I+(AG)^gAF)^{-1}(AG)^g\}\\
&\ \, \times(I+(A\tilde{G})^gA\tilde{F})^{-1}(A\tilde{G})^gA\tilde{F}(I-AA_{T,S}^{(2)})(I+AF(AG)^g)^{-1}.
\end{align*}
By Lemma \ref{3L1} and Lemma \ref{3L2}, we have
\begin{align*}
\|A_{T^\prime,S^\prime}^{(2)}&-A_{T,S}^{(2)}\|\leq \|A_{T^\prime,S^\prime}^{(2)}-A_{T^\prime,S}^{(2)}\|+\|A_{T^\prime,S}^{(2)}-A_{T,S}^{(2)}\| \\
&\leq \frac{(1+\|A\|\|A_{T^\prime,S}^{(2)}\|)\hat{\delta} (S^\prime,S)}{1-\|A\|\|A_{T^\prime,S}^{(2)}\|\hat{\delta} (S^\prime,S)  }\|A_{T^\prime,S}^{(2)}\|
+\frac{(1+ \kappa )\hat{\delta}(T,T^\prime )}{1-(1+ \kappa )\hat{\delta}(T,T^\prime)}\|A^{(2)}_{T,S}\| \\
&\leq \Big[\frac{(1+ \kappa )\hat{\delta} (S^\prime,S)(1-\hat{\delta}(T,T^\prime ))}
{\{1-(1+ \kappa )\hat{\delta}(T,T^\prime )- \kappa \hat{\delta} (S^\prime,S)\}\{1-(1+ \kappa )
\hat{\delta}(T,T^\prime )\}} \\
&\ \,+\frac{(1+ \kappa )\hat{\delta}(T,T^\prime )}{1-(1+ \kappa )\hat{\delta}(T,T^\prime)}\Big]\|A^{(2)}_{T,S}\| \\
&=\frac{(1+ \kappa )(\hat{\delta}(T,T^\prime)+\hat{\delta} (S^\prime,S))}
{1-(1+ \kappa )\hat{\delta}(T,T^\prime )- \kappa \hat{\delta}(S^\prime,S)}\|A^{(2)}_{T,S}\|,\\
\|A_{T^\prime,S^\prime}^{(2)}\|&\leq \|A_{T^\prime,S^\prime}^{(2)}-A_{T,S}^{(2)}\|+\|A^{(2)}_{T,S}\| \\
&\leq \frac{(1+ \kappa )(\hat{\delta}(T,T^\prime)+\hat{\delta} (S^\prime,S))}
{1-(1+ \kappa )\hat{\delta}(T,T^\prime )- \kappa \hat{\delta} (S^\prime,S)}\|A^{(2)}_{T,S}\| +\|A^{(2)}_{T,S}\| \\
&\leq \frac{1+\hat{\delta} (S^\prime,S)}
{1-(1+ \kappa )\hat{\delta}(T,T^\prime )- \kappa \hat{\delta} (S^\prime,S)}\|A^{(2)}_{T,S}\|.~~~~~~\cvd
\end{align*}

\begin{lemma}\label{3L3}
Let $A,\,\bar{A}=A+E\in B(X,Y)$ and $T\subset X ,S\subset Y$ be closed subspaces such that $A_{T,S}^{(2)}$ exists.
If $\|A_{T,S}^{(2)}\|\|E\|<1$,
then
$$\bar{A}_{T,S}^{(2)}=(I+A_{T,S}^{(2)}E)^{-1}A_{T,S}^{(2)}=A_{T,S}^{(2)}(I+EA_{T,S}^{(2)})^{-1}.$$
and
$$\|\bar{A}_{T,S}^{(2)}\|\leq\frac{\|A_{T,S}^{(2)}\|}{1-\|A_{T,S}^{(2)}\|\|E\|},\quad
\|\bar{A}_{T,S}^{(2)}-A_{T,S}^{(2)}\|\leq \frac{\|A_{T,S}^{(2)}\|^2\|E\|}{1-\|A_{T,S}^{(2)}\|\|E\|}.$$
\end{lemma}
{\em Proof.}
$\|A_{T,S}^{(2)}\|\|E\|<1$ implies that $(I+A_{T,S}^{(2)}E)^{-1}$ exists. Since
$$
(I+A_{T,S}^{(2)}E)A_{T,S}^{(2)}=A_{T,S}^{(2)}(I+EA_{T,S}^{(2)}),
$$
we have
$$
(I+A_{T,S}^{(2)}E)^{-1}A_{T,S}^{(2)}=A_{T,S}^{(2)}(I+EA_{T,S}^{(2)})^{-1}.
$$
Put $B=(1+A_{T,S}^{(2)}E)^{-1}A_{T,S}^{(2)}$. Then $R(B)=R(A_{T,S}^{(2)})=T$, $N(B)=N(A_{T,S}^{(2)})=S$
and $B(A+E)B=B$. Therefore, $\bar{A}_{T,S}^{(2)}=(I+A_{T,S}^{(2)}E)^{-1}A_{T,S}^{(2)}$ with
\begin{align*}
\|\bar{A}_{T,S}^{(2)}\|&\leq\|(I+A_{T,S}^{(2)}E)^{-1}\|\|A_{T,S}^{(2)}\|
\leq\dfrac{\|A_{T,S}^{(2)}\|}{1-\|A_{T,S}^{(2)}\|\|E\|}\quad \text{and}\\
\|\bar{A}_{T,S}^{(2)}-A_{T,S}^{(2)}\|&=\|-(I+A_{T,S}^{(2)}E)^{-1}A_{T,S}^{(2)}EA_{T,S}^{(2)}\|
\leq \frac{\|A_{T,S}^{(2)}\|^2\|E\|}{1-\|A_{T,S}^{(2)}\|\|E\|}.~~~~~~\cvd
\end{align*}

We close this section by giving the perturbation analysis for $A_{T,S}^{(2)}$ when $T$, $S$ and $A$ all have small
perturbations.

\begin{theorem}\label{T7}
Let $A,\bar{A}=A+E\in B(X,Y)$ and let $T\subset X$, $S\subset Y$ be closed subspaces such that $A_{T,S}^{(2)}$ exists.
Let $T'\subset X$, $S'\subset Y$ be closed subspaces with
$\hat{\delta}(T,T^\prime)< \dfrac{1}{(1+ \kappa )^2}$ and $\hat{\delta}(S,S^\prime)< \dfrac{1}{3+ \kappa }$.
Suppose that $\|A_{T,S}^{(2)}\|\|E\|<\dfrac{2\kappa}{(1+\kappa)(4+\kappa)}$. Then
\begin{align*}
&\begin{aligned}
(1)~\bar{A}_{T^\prime,S^\prime}^{(2)}&=[1+A_{T,S}^{(2)}E+(I-A_{T,S}^{(2)}A)F(I+(AG)^gAF)^{-1}(AG)^g E\\
&\ \,+\{A_{T,S}^{(2)}+(I-A_{T,S}^{(2)}A)F(I+(AG)^gAF)^{-1}(AG)^g\}\\
&\ \, \times(I+(A\tilde{G})^gA\tilde{F})^{-1}(A\tilde{G})^gA\tilde{F}(I-AA_{T,S}^{(2)})(I+AF(AG)^g)^{-1}E]^{-1}\\
&\ \,\times [A_{T,S}^{(2)}+(I-A_{T,S}^{(2)}A)F(I+(AG)^gAF)^{-1}(AG)^g \\
&\ \,+\{A_{T,S}^{(2)}+(I-A_{T,S}^{(2)}A)F(I+(AG)^gAF)^{-1}(AG)^g\}\\
&\ \, \times(I+(A\tilde{G})^gA\tilde{F})^{-1}(A\tilde{G})^gA\tilde{F}(I-AA_{T,S}^{(2)})(I+AF(AG)^g)^{-1}].
\end{aligned}\\
&(2)~\|\bar{A}_{T^\prime,S^\prime}^{(2)}\|\leq \frac{(1+\hat{\delta} (S^\prime,S))\|A^{(2)}_{T,S}\|}
{1-(1+ \kappa )\hat{\delta}(T,T^\prime )- \kappa \hat{\delta} (S^\prime,S)-(1+\hat{\delta} (S^\prime,S))\|A^{(2)}_{T,S}\|\|E\|}.\\
&(3)~ \frac{\|\bar{A}_{T^\prime,S^\prime}^{(2)}-A_{T,S}^{(2)}\|}{\|A^{(2)}_{T,S}\|}\leq \frac{(1+ \kappa )
(\hat{\delta}(T,T^\prime)+\hat{\delta} (S^\prime,S)+(1+\hat{\delta}(S^\prime,S))\|A^{(2)}_{T,S}\|\|E\|}
{1-(1+ \kappa )\hat{\delta}(T,T^\prime )- \kappa \hat{\delta} (S^\prime,S)-(1+\hat{\delta} (S^\prime,S))\|A^{(2)}_{T,S}\|\|E\|}.
\end{align*}
where, $F=\tilde{G}-G,\,\tilde{F}=\tilde{H}-\tilde{G}$ and $G,\tilde{G},\tilde{H} \in B(Y,X)$ are arbitrary such that
$R(G)=T$, $R(\tilde{G})=R(\tilde{H})=T^\prime$, $N(G)=N(\tilde{G})=S$ and $N(\tilde{H})=S^\prime$.
\end{theorem}

{\em Proof.}
We have $A_{T^\prime,S^\prime}^{(2)}$ exists and
$\|A_{T^\prime,S^\prime}^{(2)}\|\leq \dfrac{(1+\hat{\delta} (S^\prime,S))\|A^{(2)}_{T,S}\|}
{1-(1+ \kappa )\hat{\delta}(T,T^\prime )- \kappa \hat{\delta} (S^\prime,S)}$ by Theorem \ref{T3}. Thus,
$\|A_{T^\prime,S^\prime}^{(2)}\|\|E\|<\dfrac{(1+\kappa)(4+\kappa)\|A^{(2)}_{T,S}\|\|E\|}{2\kappa}<1$ and hence
$\bar{A}_{T^\prime,S^\prime}^{(2)}$ exists with
$\bar{A}_{T^\prime,S^\prime}^{(2)}=(I+A_{T^\prime,S^\prime}^{(2)}E)^{-1}A_{T^\prime,S^\prime}^{(2)}$
by Lemma \ref{3L3}. It follows from Lemma \ref{3L1} and \ref{3L2} that
\begin{align*}
\bar{A}_{T^\prime,S^\prime}^{(2)}&=[I+A_{T,S}^{(2)}E+(I-A_{T,S}^{(2)}A)F(I+(AG)^gAF)^{-1}(AG)^g E\\
&\ \,+\{A_{T,S}^{(2)}+(I-A_{T,S}^{(2)}A)F(I+(AG)^gAF)^{-1}(AG)^g\}\\
&\ \, \times(I+(A\tilde{G})^gA\tilde{F})^{-1}(A\tilde{G})^gA\tilde{F}(I-AA_{T,S}^{(2)})(I+AF(AG)^g)^{-1}E]^{-1}\\
&\ \,\times [A_{T,S}^{(2)}+(I-A_{T,S}^{(2)}A)F(I+(AG)^gAF)^{-1}(AG)^g \\
&\ \,+\{A_{T,S}^{(2)}+(I-A_{T,S}^{(2)}A)F(I+(AG)^gAF)^{-1}(AG)^g\}\\
&\ \, \times(I+(A\tilde{G})^gA\tilde{F})^{-1}(A\tilde{G})^gA\tilde{F}(I-AA_{T,S}^{(2)})(I+AF(AG)^g)^{-1}].
\end{align*}
Furthermore,
\begin{align*}
\|\bar{A}_{T^\prime,S^\prime}^{(2)}\|&\leq\frac{\|A_{T^\prime,S^\prime}^{(2)}\|}{1-\|A_{T^\prime,S^\prime}^{(2)}\|\|E\|} \\
&\leq \frac{(1+\hat{\delta} (S^\prime,S))\|A^{(2)}_{T,S}\|}
{1-(1+ \kappa )\hat{\delta}(T,T^\prime )- \kappa \hat{\delta} (S^\prime,S)-(1+\hat{\delta} (S^\prime,S))\|A^{(2)}_{T,S}\|\|E\|}.
\end{align*}
Notice that
\begin{align*}
\bar{A}_{T^\prime,S^\prime}^{(2)}-A_{T,S}^{(2)}&=(I+A_{T^\prime,S^\prime}^{(2)}E)^{-1}A_{T^\prime,S^\prime}^{(2)}-A_{T,S}^{(2)}\\
&=(I+A_{T^\prime,S^\prime}^{(2)}E)^{-1}(A_{T^\prime,S^\prime}^{(2)}-(I+A_{T^\prime,S^\prime}^{(2)}E)A_{T,S}^{(2)})\\
&=(I+A_{T^\prime,S^\prime}^{(2)}E)^{-1}(A_{T^\prime,S^\prime}^{(2)}-A_{T,S}^{(2)}-A_{T^\prime,S^\prime}^{(2)}EA_{T,S}^{(2)}).
\end{align*}
Thus we have
\begin{align*}
\|\bar{A}_{T^\prime,S^\prime}^{(2)}\!-\!A_{T,S}^{(2)}\|\!
&\leq\! \|(I+A_{T^\prime,S^\prime}^{(2)}E)^{-1}\|(\|A_{T^\prime,S^\prime}^{(2)}-A_{T,S}^{(2)}\|+\|A_{T^\prime,S^\prime}^{(2)}EA_{T,S}^{(2)}\|)\\
\!&\leq\! \frac{1}{1-\|A_{T^\prime,S^\prime}^{(2)}\|\|E\|}(\|A_{T^\prime,S^\prime}^{(2)}-A_{T,S}^{(2)}\|+\|A_{T^\prime,S^\prime}^{(2)}\|\|E\|\|A_{T,S}^{(2)}\|)\\
\!&\leq\! \frac{(1\!+\!\kappa )(\hat{\delta}(T,T^\prime)\!+\!\hat{\delta} (S^\prime,S)\!+\!(1\!+\!\hat{\delta} (S^\prime,S))\|A^{(2)}_{T,S}\|\|E\|}
{1\!-\!(1\!+\!\kappa )\hat{\delta}(T,T^\prime )\!-\!\kappa \hat{\delta} (S^\prime,S)\!-\!(1\!+\!\hat{\delta} (S^\prime,S))\|A^{(2)}_{T,S}\|\|E\|}
\|A^{(2)}_{T,S}\|.~~~~\cvd
\end{align*}

\vskip0.2cm
{\bf{Acknowledgement.}} The authors thank to referees  for their helpful comments and suggestions.

\end{document}